\documentclass[11pt]{article}
\usepackage{color,float,soul,graphicx}%
\usepackage{multirow,enumerate,enumitem}
\usepackage{url,cite,fullpage,fancyhdr}
\usepackage{amsmath,amssymb,amsfonts,amsthm,mathrsfs,amscd}%
\usepackage[doc]{optional}
\usepackage{mathtools,esvect}
\usepackage{cite}
\usepackage{color}
\usepackage{float}
\usepackage{subcaption}
\usepackage{soul}
\usepackage{xcolor,textcomp,manyfoot}%
\usepackage{booktabs}%
\usepackage{algorithm,algorithmicx}%
\usepackage{algpseudocode,listings}%
\usepackage{fancyhdr}
\newcommand{\fenv}[1]%
{\ensuremath{\,\overrightarrow{\operatorname{env}}_{#1}}}
\newcommand{\benv}[1]%
{\ensuremath{\,\overleftarrow{\operatorname{env}}_{#1}}}

\renewcommand\theenumi{(\roman{enumi})}

\theoremstyle{thmstyleone}%
\theoremstyle{thmstyletwo}%
\theoremstyle{thmstylethree}%
\newtheorem{theorem}{Theorem}
\newtheorem{lemma}[theorem]{Lemma}
\newtheorem{corollary}[theorem]{Corollary}
\newtheorem{proposition}[theorem]{Proposition}%
\newtheorem{definition}{Definition}%

\usepackage{cite}
\theoremstyle{definition}

\newcommand{\nexto}{\kern -0.54em}

\newcommand{\dZ}{{\cal Z \kern -0.7em Z}}
\newcommand{\dC}{{\rm\hbox{C \kern-0.8em\raise0.2ex\hbox{\vrule height5.4pt width0.7pt}}}}
\newcommand{\dQ}{{\rm\hbox{Q \kern-0.85em\raise0.25ex\hbox{\vrule height5.4pt width0.7pt}}}}

\usepackage{epsfig}
\usepackage{latexsym}

\newcommand{\RR}{\mathbb{R}}

\date{}
\begin{document}
\title{Translation Mappings of Quasimonotonicity Beyond Smoothness 
}
\author{Oday Hazaimah
\footnote{E-mail: {\tt odayh982@yahoo.com}. https://orcid.org/0009-0000-8984-2500. 
St. Louis, MO, USA}}
%

\maketitle\thispagestyle{fancy}
\maketitle
\begin{abstract}\noindent Monotonicity of a mapping implies its pseudomonotonicity and hence quasimonotonocity, the converse is not true. In this note we intend to study the situations under which quasimonotonicity of a mapping implies its monotonicity. Thus we generalize some results in the literature related to the connection between monotonocity and its generalized classes for multi-valued mappings via translation maps in real topological spaces. No differentiability assumption is required but continuity assumption is imposed. 
\medskip 

\noindent {\textbf{Keywords}:}  Generalized Jacobian; Subgradients; Monotone Operators and Generalizations.
\medskip 

\noindent \textbf{\small Mathematics Subject Classification:}{ 47H04, 47H05, 47H07}
\end{abstract}

\maketitle

\section{Introduction}
Monotone maps are fundamental components in nonlinear analysis and can be considered as a generalizations of convex functions when gradient maps exist. Monotonicity and convexity are closely related. It is well established that convex functions can be characterized by their Clarke generalized subdifferential maps: namely, subgradients of real-valued convex functions have the property of being monotone (see \cite{Clarke,Rockafellar}). This fact has been considered for lower semi continuous functions and for locally Lipschitzian functions on $n$-dimensional Euclidean spaces (see, for more details \cite{Clarke,Komlosi,Somayeh}). However, according to the emerging needs appearing in many applications as in optimization and economics, it is highly desired to proceed beyond monotonicity. Notions of pseudomonotone and quasimonotone operators have been introduced in the literature to consider a broader class than monotonicity see, for instance, \cite{Karamardian,KaramSch,Komlosi} and the references therein. Generalized monotone maps play a crucial role in applications and they address uniqueness and stability questions in the general theory of economic equilibrium. For instance, the class of quasivariational inequalities (an extension of variational inequalities where the constraint set is a moving set) defined over infinite dimensional Banach space, characterizes the dynamic competitive equilibrium problem for time-dependent pure exchange economy in which pseudomonotonicity assumption is imposed on the map defining the inequalities (see \cite{Sultana}).
It is well known that a monotone mapping is pseudomonotone, but the converse is not generally true. Similarly, any pseudomonotone operator is quasimonotone, but the converse is not true. Due to these connections an interesting investigation address questions like when and under what conditions a pseudomonotone or quasimonotone operators coincide with monotone maps. Generalized monotonicity plays the role of monotonicity without losing important properties for the analysis and solution in many classical models. 
It is known that \cite{KaramSch} differentiable generalized monotone maps in terms of the Jacobian can be extended to nondifferentiable locally Lipschitz maps using the generalized Jacobian in the sense of Clarke \cite{Clarke}.
A characterization of pseudoconvex functions in terms of the gradient without using function values was derived by Karamardian \cite{Karamar}, and later he used his characterization of pseudoconvex functions to introduce the notion of pseudomonotonicity for maps that are not necessarily the gradient of a function. Numerous characterizations have been proposed for certain types of generalized convex functions, see e.g. \cite{Komlosi,Luc,Schaible}. 
The smooth characterizations in \cite{KaramSch} for quasimonotone and pseudomonotone maps have been extended by Luc and Schaible \cite{Schaible} to the nondifferentiable case. If $F$ is a quasimonotone map then the translation map $F+\omega$ is not quasimonotone for any $\omega$ belongs to the dual space of a topological vector space $X$. In the case of a single-valued, linear map $F$ defined on the whole space $\RR^n$, it is known for instance that if $F+\omega$ is quasimonotone, then $F$ is monotone \cite{KaramSch}. An extension of this result has been considered and studied, by He \cite{He}, into a more general case when $F$ is a nonlinear map defined on the Euclidean space $\RR^n$. While Isac and Motreanu \cite{IsacMotreanu} tackled the same result when $X$ is an infinite-dimensional Hilbert space with nonempty interior convex domains. If $F$ is a single-valued, continuous map which is Gateaux differentiable on the interior convex domain such that $F+\omega$ is quasimonotone, 
then $F$ is monotone. Differentiability is essential in the previous literature since the arguments were based on first-order characterizations of generalized monotonicity \cite{John}. The work of Hadjisavvas \cite{Hadjisavvas}  extended the result of interest to set-valued maps defined on convex sets in topological vector spaces in which the differentiability assumption was relaxed. The authors in \cite{Schaible} characterized nonsmooth monotone maps in terms of the generalized Jacobian for convex sets in finite dimensional Euclidean space and extended necessary and sufficient conditions to a variety of generalized monotone classes (quasi- /pseudo- /strictly /strongly monotone maps). 
\medskip

The purpose of this note is to draw some concerns between monotonicity and quasimonotonicity by translation maps with the absence of smoothness in topological vector spaces in attempt to broaden the theory of generalized monotone maps. Thus the present paper aims to extend a result by \cite{Schaible} from finite-dimensional Euclidean spaces to real topological vector spaces for nonsmooth maps; and to complete a one-sided implication result by \cite{Hadjisavvas} for multivalued maps defined on convex sets. No assumption of differentiability has been made, and the domain does not have a nonempty interior, thus we can apply the result broadly to any vector space with a partial order on its positive elements such as the positive cones of $l_+^2, L_+^2$, the set of positive semidefinite matrices and the nonnegative orthant $\RR^n_+$.

\section{Preliminaries} 
In this section, we provide a short review on the types of generalized monotone maps and recall some basic properties and notations from the theory of nonsmooth maps which can be found in details in \cite{Clarke}.
Let $X$ be a topological vector space and $X^*=L(X,\RR)$ its dual space (The set of all linear maps from $X$ to $\RR$). Let $\Omega$ be a nonempty convex subset of $X$, and $F:\Omega\rightrightarrows X^*$ is a set-valued map. Given $x,y\in X$, 
we denote by $[x,y]$ the line segment $\{(1-\lambda)x+\lambda y: \lambda\in [0,1]\}$. A straight line $L$ in the dual space is defined by $L=\{u+\lambda v:\lambda\in\RR\}$ such that $u\in X^*, v\in X^*\setminus \{0\}$. For any vector in the dual space $v\in X^*$, is said to be orthogonal to $\Omega$ if $\langle v,x\rangle=\langle v,y\rangle$ for all $x,y\in\Omega$. $L$ is orthogonal to $\Omega$ if $v$ is orthogonal to $\Omega$.
The generalized Jacobian matrix (subdifferential mapping) of $F$ at $x\in\Omega$ is defined by 
\medskip 

$$\partial F(x):=conv\Big\{\lim_{i\to\infty} \ JF(x_i) : x_i\to x, F \ \text{is differentiable at} \ x_i\Big\},$$ 
where \textit{conv} denotes the convex hull and $JF(x)$ is the Jacobian of $F$ at $x,$ i.e., the convex hull of all points on the form $\lim JF$ where the sequence $\{x_i\}$ is arbitrary and converges to $x$. If $F$ is a single-valued convex function then $\partial F$ is the subdifferential operator which is always exist. If $F$ is convex then $\partial F$ is a maximal monotone mapping. If $F$ is differentiable gradient mapping then $\partial F$ is precisely the gradinet maping $\{\nabla F\}$. If $F$ is locally Lipschitz map then the set $\partial F(x)$ is nonempty compact. Next definition draws properties and relations between monotone maps and generalized monotone (i.e., psedumonotone and quasimonotone) maps. 
\begin{definition}\label{monotonedefi} Let $\Omega$ be an arbitrary subset of $X$. The mapping $F:\Omega\to X$, is said to be:
\begin{itemize}
\item [(i)] Monotone if for all $x, y \in\Omega$,
$$\langle F(x)-F(y), x-y\rangle \geq 0.$$ 
\item [(ii)] Strictly monotone if the above inequality is strict for all $x\not= y$ in $\Omega.$
\item [(iii)] Strongly monotone if there exists a modulus $\lambda > 0$ such that for all $x, y\in\Omega,$
$$\langle F(x)-F(y), x-y\rangle\geq\lambda\|x-y\|^2.$$
\item [(iv)] Pseudomonotone (in the sense of Karamardian \cite{Karamardian}) if for all $x, y\in\Omega,$
$$\langle F(y), x-y\rangle\geq 0\implies \langle F(x), x-y\rangle\geq 0.$$
\item [(v)] Quasimonotone (in the sense of Karamardian \cite{Karamardian}) if for all $x, y\in\Omega,$
$$\langle F(y), x-y\rangle > 0\implies \langle F(x), x-y\rangle\geq 0.$$
\end{itemize}
\end{definition}
We will introduce the following useful lemma and proposition that we need before we proceed to the main result.
\begin{lemma}\label{monotonelemma}
Let $F:\Omega\rightrightarrows X^*$ be a set-valued map on a nonempty convex subset $\Omega$ of a topological vector space $X$. Assume that there exists a straight line $L\subset X^*$, that is not orthogonal to $\Omega$, such that for every $\omega\in L$, $F+\omega$ is quasimonotone. Then the restriction of $F$ on any line segment of $\Omega$ which is not orthogonal to $L$ is monotone.
\end{lemma}
\begin{proposition}\label{prop}
Let $\Omega$ be a convex closed subset of $X$ and $F:\Omega\rightrightarrows X^*$ a continuous mapping.
Then the following statements are equivalent:
\begin{enumerate}
\item $F$ is monotone on $\Omega$,
\item there exists a straight line $L\subset X^*$, that is not orthogonal to $\Omega$, such that the mapping $F+\omega$ is pseudomonotone for any $\omega\in L$,
\item there exists a straight line $L\subset X^*$, that is not orthogonal to $\Omega$, such that the mapping $F+\omega$ is quasimonotone for any $\omega\in L.$
\end{enumerate}
\end{proposition}
\begin{proof}
The two implications $(i)\implies (ii)$ and $(ii)\implies (iii)$ is a straightforward consequence of Definition \ref{monotonedefi}. For the remaining implication $(iii)\implies (i)$, we invoke Lemma \ref{monotonelemma} and monotonocity for the restricted $F$ on any line segment of $\Omega$ is guaranteed. Since $\Omega$ is closed, the continuity of $F$ implies that $F$ is monotone on $\Omega$.
\end{proof}

\section{Main Results}
We introduce the main result of this note. The proof of the main theorem basically has constructed by employing the notion of monotonicity for the broader class of topological spaces.
\begin{theorem}\label{mainthm}
Let $X$ be a real topological vector space, and let $\Omega$ be a nonempty convex subset of $X$, and $F:\Omega\rightrightarrows X^*$ be a set-valued map. Then there exists a straight line $L\subset X^*$, that is not orthogonal to $\Omega$, such that for every $\omega\in L$, the translation map $F+\omega$ is quasimonotone if and only if the subdifferential mapping $\partial F(x)$ is positive semidefinite for every $x\in\Omega$.  
\end{theorem}
\begin{proof}
We first begin by the "only if" part.
Let $x,y\in\Omega$. If $\langle v,x-y\rangle\not=0$ then $L$ is not orthogonal to $[x,y]$; thus by Lemma \ref{monotonelemma}, $F$ is monotone on $[x,y]$ and hence quasimonotone. 
Now assume that $\langle v,x-y\rangle=0$. Since $L$ is not orthogonal to $\Omega$, there exists $z\in\Omega$ such that $\langle v,z\rangle\not=\langle v,x\rangle$ and $\langle v,z\rangle\not=\langle v,y\rangle.$ Let the parameter $\alpha\in(0,1)$ and for all $\alpha$ define the averaged point $z_{\alpha}=\alpha z+(1-\alpha)\displaystyle\frac{x+y}{2}.$ Then $L$ is not orthogonal to the line segments $[x,z_{\alpha}], \ [y,z_{\alpha}]$ and $[\frac{x+y}{2},z]$; thus its restriction to each of these three line segments is monotone by Lemma \ref{monotonelemma}. We conclude that the following inclusions; $x^*\in F(x), y^*\in F(y), z^*\in F(z)$ and $z^*_{\alpha}\in F(z_{\alpha})$ coincide with the following inequalities: 
\begin{enumerate}
\item $\langle z_{\alpha}^*,z_{\alpha}-x\rangle\geq\langle x^*,z_{\alpha}-x\rangle$
\item $\langle z_{\alpha}^*,z_{\alpha}-y\rangle\geq\langle y^*,z_{\alpha}-y\rangle$
\item $\langle z^*,z-z_{\alpha}\rangle\geq\langle z^*_{\alpha},z-z_{\alpha}\rangle$
\end{enumerate}
It is clear, from $(i)$ and $(ii)$, that $\langle z_{\alpha}^*,2z_{\alpha}-(x+y)\rangle\geq\langle x^*,z_{\alpha}-x\rangle+\langle y^*,z_{\alpha}-y\rangle,$ this inequality can be rewritten as 
\begin{equation}\label{(i)+(ii)}
2\alpha\big\langle z_{\alpha}^*,z-\frac{x+y}{2}\big\rangle\geq\langle x^*,z_{\alpha}-x\rangle+\langle y^*,z_{\alpha}-y\rangle.
\end{equation}
Moreover, $(iii)$ can be written as 
\begin{equation*}
\big\langle z^*,(1-\alpha)\big(z-\frac{x+y}{2}\big)\big\rangle\geq\big\langle z^*_{\alpha},(1-\alpha)\big(z-\frac{x+y}{2}\big)\big\rangle.
\end{equation*}
Consequently, and since $(1-\alpha)$ is positive, we obtain 
\begin{equation}\label{(iii)}
\big\langle z^*,\big(z-\frac{x+y}{2}\big)\big\rangle\geq\big\langle z^*_{\alpha},\big(z-\frac{x+y}{2}\big)\big\rangle
\end{equation}
It follows, by combining \eqref{(i)+(ii)} and \eqref{(iii)}, that 
$$2\alpha\big\langle z^*,z-\frac{x+y}{2}\big\rangle\geq\langle x^*,z_{\alpha}-x\rangle+\langle y^*,z_{\alpha}-y\rangle.$$ 
Now consider the case when $\alpha$ is sufficiently small and approaching the lower bound $0$. Then $\displaystyle\lim_{\alpha\to 0}z_{\alpha}=\frac{x+y}{2}$, and $\langle x^*-y^*,y-x\rangle\leq 0$ which means $F$ is monotone on the line segment $[x,y]$ for all $x,y\in\Omega$ and all $x^*\in F(x),\ y^*\in F(y)$. Since the interval $[x,y]$ is arbitrary in $\Omega$ then $F$ is monotone on $\Omega$. Using the definition, by Clarke \cite{Clarke}, of the generalized Jacobian subdifferential, it is only required to prove that the Jacobian $JF$ is positive semidefinite whenever it exists. Given $x\in\Omega$ and assume that the Jacobian $JF(x)$ exists. Suppose, on the contrary, that $JF(x)$ is not positive semidefinite. Then there exist $u\in X$ such that $\langle u,JF(x)u\rangle<0$. Hence by the definition of the Jacobian
$$JF(x)u:=\lim_{t\to 0}\frac{F(x+tu)-F(x)}{t}$$
and for sufficiently small $t>0$, we have $\langle u,F(x+tu)-F(x)\rangle<0.$ Setting $y=x+tu$ and for a fixed small $t\in (0,1)$, we have 
$$\langle y-x,F(y)-F(x)\rangle <0$$
which is a contradiction to the monotonicity of $F.$ Thus the generalized Jacobian (i.e., the subdifferential operator) is positive semidefinite. Hence, the subgradients (elements of subdifferential mapping) are positive semidefinite. Suppose that we are given two points $x,y\in L$, by the mean value theorem \cite{Clarke}, we have 
$$F(x)-F(y)\in conv \big\{\partial F([x,y])(x-y)\big\}.$$
Thus 
$$\langle F(x)-F(y),x-y\rangle\in conv \big\{\langle x-y,G(x-y)\rangle:G\in\bigcup_z\partial F(z)\big\}$$
where $G$ is a subgradient and $z$ is a convex combination of $x$ and $y.$ 
Since every element of the above convex hull is nonnegative, therefore we have 
$\langle F(x)-F(y),x-y\rangle\geq 0$, which means $F$ is monotone on $\Omega$. By invoking Proposition \ref{prop}, particularly the implication $(i)\implies (iii)$, then 
there exists a straight line $L\subset X^*$ such that the mapping $F+\omega$ is quasimonotone for any $\omega\in L.$
\end{proof}
It is worth mentioning that there is only a sufficient condition when $F$ is strictly monotone but not necessary, as in the differentiable case. The argument is very similar to the previous proof except that it yields the positive definiteness which is a special case of positive semidefinitness. This implies the following consequent result.
\begin{corollary}
The assumptions are as in Theorem \ref{mainthm}. Then there exists a straight line $L\subset X^*$, that is not orthogonal to $\Omega$, such that for every $\omega\in L$, the translation map $F+\omega$ is quasimonotone if and only if the subdifferential mapping $\partial F(x)$ is positive definite for every $x\in\Omega$.  
\end{corollary}
Notice that the assumption "$L$ is not orthogonal to $\Omega$" is necessary condition and hence cannot be relaxed as if we, for example, consider the space $X=\RR^3$ and the set $\Omega=\RR^2\times\{0\}$ along with the line $L=\{0\}\times\RR^2$. Define the map $F:\RR^2\times\{0\}\to\RR^3$ by $(x,y,0)\mapsto e^{-(x+y)}(1,1,0)$, it is concluded that $F$ is not monotone whereas $F+\omega$ is pseudomonotone for all lines $\omega\in L$. Hence even if we reduce the quasimonotonocity in Theorem \ref{mainthm} to pseudomonotonocity for all $\omega\in L$, the statement of $L$ is not orthogonal to $\Omega$ is still necessary. Since every pseudomonotone mapping is quasimonotone, then the following statement directs immediately from Theorem \ref{mainthm}.
\begin{corollary}
Let $X$ be a real topological vector space, and let $\Omega$ be a nonempty convex subset of $X$, and $F:\Omega\rightrightarrows X^*$ be a set-valued map. Assume that there exists a straight line $L\subset X^*$, that is not orthogonal to $\Omega$, such that for every $\omega\in L$, the $F+\omega$ is pseudomonotone. Then $\partial F(x)$ is positive semidefinite for every $x\in\Omega$.  
\end{corollary}



\medskip

\subsection*{Declarations} 
The author declares that there was no conflict of interest or competing interest. This work was done without any resource of funding. Data Availability is not applicable.



\end{document}